\documentclass[11pt,a4paper]{article}

\usepackage{mathrsfs}
\usepackage{amscd}
\usepackage{hyperref}
\hypersetup{colorlinks=true,citecolor=green,linkcolor=blue,urlcolor=blue}
\usepackage{tipa}
\usepackage{amssymb}
\usepackage{amsfonts}
\usepackage{stmaryrd}
\usepackage{amssymb}

\usepackage{CJK}  % 中文支持宏包
\usepackage{indentfirst} % 首行缩进宏包

\usepackage{latexsym}   % 用来排数学公式
\usepackage{bm}         % 处理数学公式中的黑斜体的宏包

\usepackage{pifont}
\usepackage{bm}
\usepackage{amsmath,amssymb,amsfonts}% AMSLaTeX宏包 用来排出更加漂亮的公式
\usepackage{indentfirst}

\usepackage{xcolor}
\usepackage{cases}
\usepackage{wasysym}
\usepackage{amsthm}
\newcommand{\hei}{\CJKfamily{hei}}      % 黑体   (Windows自带simhei.ttf)
\newtheorem{theorem}{Theorem}[section]
\newtheorem{lemma}[theorem]{Lemma}
\newtheorem{proposition}[theorem]{Proposition}
\newtheorem{conjecture}[theorem]{Conjecture}
\newtheorem{remark}[theorem]{Remark}
\newtheorem{corollary}[theorem]{Corollary}
\newtheorem{definition}[theorem]{Definition}
\CJKtilde   %用于解决英文字母和汉字的间距问题。例如：变量~$x$~的值。
\begin{document}

\title{\Large \hei \textbf{On the pro-$p$-Iwahori invariants of supersingular representations of unramified $U(2, 1)$}}
\date{}
\author{\textbf{Peng Xu}}
\maketitle

\begin{abstract}
Let $G$ be the unramified unitary group $U(2, 1)(E/F)$ defined over a non-archimedean local field $F$ of odd residue characteristic $p$. In this paper, for any supersingular representation of $G$ that contains the Steinberg weight, we prove its pro-$p$-Iwahori invariants, as a right module over the pro-$p$-Iwahori--Hecke algebra of $G$, is \emph{not} simple.
\end{abstract}

\tableofcontents

\section{Introduction}

Let $G$ be the unramified unitary group $U(2, 1)(E/F)$ defined over a non-archimedean local field $F$ of odd residue characteristic $p$, and $K$ be a maximal compact open subgroup of $G$. Let $I_K$ be the Iwahori subgroup in $K$, and $I_{1, K}$ be the pro-$p$-Sylow subgroup of $I_K$.

In this paper, we prove:
\begin{theorem}\label{main theorem: intro}(Theorem \ref{main theorem for st})
Let $\pi$ be a supersingular representation of $G$ containing the Steinberg weight of $K$. Then, we have
\begin{center}
$\textnormal{dim}_{\overline{\mathbf{F}}_p} \pi^{I_{1, K}} \geq 2$.
\end{center}
\end{theorem}
Here, an irreducible smooth $\overline{\mathbf{F}}_p$-representation of $G$ is supersingular, if it is a quotient of certain spherical universal modules, see \ref{subsec: ss} for the precise definition.

Indeed, what we proved in Theorem \ref{main theorem for st} is conceptually stronger than that stated in Theorem \ref{main theorem: intro}: as a counterpart to a result of Barthel--Livn$\acute{\text{e}}$ on $GL_2$ (Theorem \ref{GL_2}), we propose a naive conjecture (\ref{main conjeture}) on the pro-$p$-Iwahori invariants of supersingular representations of $G$, and we prove it for those containing the \emph{Steinberg} weight.

\begin{remark}
Any spherical universal module of $G$ is infinite dimensional (\cite[Corollary 4.6]{X2016}). Therefore, supersingular representations of $G$ containing a given weight do exist. Moreover, according to the philosophy of Breuil and Pa$\check{\text{s}}$k$\bar{\text{u}}$nas on $GL_2$ (\cite[Proposition 10.2]{BP12}), it is indeed very likely that there are infinitely many non-isomorphic supersingular representations of $G$ containing a given weight.
\end{remark}

Our motivation of this paper is to understand the pro-$p$-Iwahori invariants of irreducible smooth $\overline{\mathbf{F}}_p$-representations of $G$. Let $\mathcal{H} (I_{1, K}, 1) := \text{End}_G (\text{ind}^G _{I_{1, K}} 1)$ be the pro-$p$-Iwahori--Hecke algebra of $G$. For a smooth $\overline{\mathbf{F}}_p$-representation $\pi$ of $G$, its pro-$p$-Iwahori invariants $\pi^{I_{1, K}}$ is naturally a right module over $\mathcal{H} (I_{1, K}, 1)$. On one hand, when $\pi$ is an irreducible subquotient of principal series, we understand it well: $\pi^{I_{1, K}}$ is \emph{simple} over $\mathcal{H}(I_{1, K}, 1)$ (\cite[Corollary 4.4]{Karol-Peng2012}). On the other hand, when $\pi$ is a supersingular representation, we know very little about $\pi^{I_{1, K}}$. The following corollary deduced from Theorem \ref{main theorem: intro}, which, besides other things, unveils a new feature of them.

\begin{corollary}\label{non-simpleness: intro}(Corollary \ref{non-simpleness})
Let $\pi$ be a supersingular representation of $G$ containing the Steinberg weight of $K$. Then, the space $\pi^{I_{1, K}}$ as a right module over $\mathcal{H} (I_{1, K}, 1)$, is \textbf{not} simple.
\end{corollary}

\medskip

We now review briefly what is known in the literatures (to our knowledge). For a two dimensional continuous irreducible generic $\overline{\mathbf{F}}_p$-representation of the absolute Galois group $G_{\mathbf{Q}_{p^f}} (f> 1)$, Breuil and Pa$\check{\text{s}}$k$\bar{\text{u}}$nas have associated it an infinite family of supersingular representations of $GL_2 (\mathbf{Q}_{p^f})$ (\cite[Theorem 1.5]{BP12}), and one of these supersingular representation has been proved by Emerton--Gee--Savitt (\cite{EGS15}) that it will \emph{appear} globally, which implies that for such a supersingular representation its pro-$p$-Iwahori invariants is equal to the original recipe used to construct it. These in all suggest\footnote{We heard of this from Prof. Pa$\check{\text{s}}$k$\bar{\text{u}}$nas several years ago.} that the pro-$p$-Iwahori invariants of supersingular representations of $GL_2 (F)$, as right modules over the its pro-$p$-Iwahori--Hecke algebra, might \emph{not} be simple in general, unless $F= \mathbf{Q}_p$.

Our Corollary \ref{non-simpleness: intro} is in a different manner from the result on $GL_2 (\mathbf{Q}_{p^f})$ ($f> 1$) mentioned above, and the condition of requiring to contain the Steinberg weight, as remarked above, is likely satisfied by infinitely many non-isomorphic supersingular representations of $G$.

\medskip
The paper is organized as follows. In section \ref{sec: notations and pre}, we fix notations and review some preliminaries facts that we will use. In section \ref{sec: conj}, we propose a naive conjecture on the pro-$p$-Iwahori invariants of supersingular representations of $G$. In section \ref{sec: the image of pro-p under two operators}, we describe the image of the pro-$p$-Iwahori invariants of a maximal compact induction under certain group operators. In section \ref{sec: proof of main results}, we prove our main results.

\section{Notations and Preliminaries}\label{sec: notations and pre}

\subsection{General notations}
Let $F$ be a non-archimedean local field of odd residue characteristic $p$, with ring of integers $\mathfrak{o}_F$ and maximal ideal $\mathfrak{p}_F$, and $k_F$ be its residue field of cardinality $q=p^f$. Fix a separable closure $F_s$ of $F$. Let $E$ be the unramified quadratic extension of $F$ in $F_s$. We use similar notations $\mathfrak{o}_E$, $\mathfrak{p}_E$, $k_E$ for analogous objects of $E$. Fix a uniformizer $\varpi_{E}$ in $E$. We equip $E^3$ with the Hermitian form h:
\begin{center}
 $\text{h}:~E^3 \times E^3 \rightarrow E$,
$(v_{1}, v_{2}) \mapsto ~v_{1}^{\text{T}}\beta \overline{v_2}, v_{1}, v_{2}\in E^3$.
\end{center}
Here, $-$ denotes the non-trivial Galois conjugation on $E/F$, inherited by $E^3$, and
$\beta$ is the matrix
\[ \begin{matrix}\begin{pmatrix} 0  & 0 & 1  \\ 0  & 1 & 0\\
1 & 0 & 0
\end{pmatrix}
\end{matrix}. \]
The unitary group $G$ is defined as:

\begin{center}
$G=\{g\in \text{GL}(3, E)\mid \text{h}(gv_1, gv_2)= \text{h}(v_1, v_2), \forall~v_1, v_2~\in E^3\}.$
\end{center}

Let $B=HN$ (resp, $B'= HN'$) be the subgroup of upper (resp, lower) triangular matrices of $G$, with $N$ (resp, $N'$) the unipotent radical of $B$ (resp, $B'$) and $H$ the diagonal subgroup of $G$. Denote an element of the following form in $N$ and $N'$ by $n(x, y)$ and $n'(x, y)$ respectively:
\begin{center}
$\begin{pmatrix}  1 & x & y  \\ 0 & 1 & -\bar{x}\\
0 & 0 & 1
\end{pmatrix}$, ~
$\begin{pmatrix}  1 & 0 & 0   \\ x & 1 & 0\\
y & -\bar{x} & 1
\end{pmatrix},$
\end{center}
where $(x, y)\in E^2$ satisfies $x\bar{x}+ y+ \bar{y}=0$. For any $k\in \mathbb{Z}$, denote by $N_k$ (resp, $N'_k$) the subgroup of $N$ (resp, $N'$) consisting of $n(x, y)$ (resp, $n'(x, y)$) with $y\in \mathfrak{p}^{k}_E$. For $x\in E^\times$, denote by $h(x)$ an element in $H$ of the following form:
\begin{center}
$\begin{pmatrix}  x & 0 & 0  \\ 0 & -\bar{x}x^{-1} & 0\\
0 & 0 & \bar{x}^{-1}
\end{pmatrix}.$
\end{center}

\medskip
Up to conjugacy, the group $G$ has two maximal compact open subgroups $K_0$ and $K_1$, given by:
\begin{center}
$K_0= \begin{pmatrix}  \mathfrak{o}_E & \mathfrak{o}_E & \mathfrak{o}_E  \\ \mathfrak{o}_E  & \mathfrak{o}_E & \mathfrak{o}_E\\
\mathfrak{o}_E & \mathfrak{o}_E & \mathfrak{o}_E
\end{pmatrix}\cap G, ~K_1= \begin{pmatrix}  \mathfrak{o}_E & \mathfrak{o}_E & \mathfrak{p}^{-1}_E  \\ \mathfrak{p}_E  & \mathfrak{o}_E & \mathfrak{o}_E\\
\mathfrak{p}_E & \mathfrak{p}_E & \mathfrak{o}_E
\end{pmatrix}\cap G.$
\end{center}
The maximal normal pro-$p$ subgroups of $K_0$ and $K_1$ are respectively:

$K^1 _0= 1+\varpi_E M_3 (\mathfrak{o}_E)\cap G, ~K^1 _1= \begin{pmatrix}  1+\mathfrak{p}_E & \mathfrak{o}_E & \mathfrak{o}_E  \\ \mathfrak{p}_E  & 1+\mathfrak{p}_E & \mathfrak{o}_E\\
\mathfrak{p}^2_E & \mathfrak{p}_E & 1+\mathfrak{p}_E
\end{pmatrix}\cap G.$

Let $\alpha$ be the following diagonal matrix in $G$:
\[ \begin{matrix}\begin{pmatrix} \varpi_{E}^{-1}  & 0 & 0  \\ 0  & 1 & 0\\
0 & 0 & \varpi_{E}
\end{pmatrix}
\end{matrix} ,\]
and put $\beta'=\beta \alpha^{-1}$. Note that $\beta\in K_0$ and $\beta'\in K_1$.  We use $\beta_K$ to denote the unique element in $K\cap \{\beta, \beta'\}$.

Let $K\in \{K_0, K_1\}$, and $K^1$ be the maximal normal pro-$p$ subgroup of $K$. We identify the finite group $\Gamma_K= K/K^1$ with the $k_F$-points of an algebraic group defined over $k_F$, denoted also by $\Gamma_K$: when $K$ is $K_0$, $\Gamma_K$ is $U(2, 1)(k_E /k_F)$, and when $K$ is $K_1$, $\Gamma_K$ is $U(1, 1)\times U(1)(k_E /k_F)$. Let $\mathbb{B}$ (resp, $\mathbb{B}'$) be the upper (resp, lower) triangular subgroup of $\Gamma_K$, and $\mathbb{U}$ (resp, $\mathbb{U}'$) be its unipotent radical.

The Iwahori subgroup $I_K$ (resp, $I'_K$) and pro-$p$ Iwahori subgroup $I_{1,K}$ (resp, $I' _{1,K}$) in $K$ are the inverse images of $\mathbb{B}$ (resp, $\mathbb{B}'$) and $\mathbb{U}$ (resp, $\mathbb{U}'$) in $K$.
Recall we have the following Bruhat decomposition for $K$:
 \begin{center}
 $K= I_K \cup I_K \beta_K I_K$.
 \end{center}

Put $H_0= H\cap I_K$, and $H_1= H\cap I_{1, K}$. As $H_0 /H_1 \cong I_K / I_{1, K}$, we will identify the characters of these groups. For a character $\chi$ of $H_0$, i.e., a character of $H_0 /H_1$, denote by $\chi^s$ the character given by $\chi^s (h):= \chi (\beta_K h \beta_K)$

\smallskip

Denote by $n_K$ and $m_K$ the unique integers such that $N\cap I_{1, K}= N_{n_K}$ and $N'\cap I_{1, K}=N'_{m_K}$. We have:
\begin{center}
$n_K +m_K =1$
\end{center}

Note that the coset spaces $N_{n_K}/N_{n_K +1}$ and $N'_{m_K}/N'_{m_K +1}$ are indeed groups of order respectively $q^{t_K}$ and $q^{4-t_K}$, for $t_K= 3$ or $1$, depending on $K$ is hyperspecial or not.

We will use the following group:
\begin{center}
$L_{q^3} := \{(x, t)\in k^2 _E \mid x\bar{x} + t +\bar{t}=0\}$,
\end{center}
and its central subgroup:
\begin{center}
$L_q := \{(0, t)\mid t +\bar{t}=0\}$.
\end{center}
Here, the group structure of $L_{q^3}$ is given by
\begin{center}
$(x, t) \cdot (x', t') := (x+x', t+ t'- x' \bar{x}).$
\end{center}

We may identify these groups naturally:
\begin{center}
$L_{n_K}: N_{n_K}/N_{n_K +1} \cong L_{q^{t_K}}$

$n(x, \varpi^{n_K} _E t) \mapsto (x\varpi^{-n_K}_E , t) (\text{mod}~\mathfrak{p}_E)$
\end{center}

\begin{center}
$L_{m_K}: N'_{m_K}/N'_{m_K +1}\cong L_{q^{4-t_K}}$

$n'(\varpi_E x, \varpi^{m_K} _E t) \mapsto (x\varpi^{2-m_K}_E, t)(\text{mod}~\mathfrak{p}_E)$
\end{center}
Here, the elements $x$ and $t$ on the left hand side lie in $\mathfrak{o}_E$.

We fix a non-zero element $\mathfrak{t}\in \mathfrak{o}^\times _E$ with trace zero.

\medskip
All the representations of $G$ and its subgroups considered in this paper are smooth over $\overline{\mathbf{F}}_p$.

\subsection{The spherical Hecke algebra $\mathcal{H}(K, \sigma)$}
 Let $K$ be a maximal compact open subgroup of $G$, and $(\sigma, W)$ be an irreducible smooth representation of $K$. As $K^1$ is pro-$p$ and normal, $\sigma$ factors through the finite group $\Gamma_K= K/K^1$, i.e., $\sigma$ is the inflation of an irreducible representation of $\Gamma_K$. Conversely, any irreducible representation of $\Gamma_K$ inflates to an irreducible smooth representation of $K$. We may therefore identify irreducible smooth representations of $K$ with irreducible representations of $\Gamma_K$, and we shall call them \emph{weights} of $K$ or $\Gamma_K$ from now on.

  It is known that $\sigma^{I_{1,K}}$ and $\sigma_{I'_{1,K}}$ are both one-dimensional, and that the natural composition map $\sigma^{I_{1,K}}\hookrightarrow \sigma \twoheadrightarrow \sigma_{I'_{1,K}}$ is an isomorphism of vector spaces (\cite[Theorem 6.12]{C-E2004}). Denote by $j_\sigma$ the inverse of the map aforementioned. For $v\in \sigma^{I_{1,K}}$, we have $j_\sigma (\bar{v})= v$, where $\bar{v}$ is the image of $v$ in $\sigma_{I'_{1,K}}$. By composition, we view $j_\sigma$ naturally a map in $\text{End}_{\overline{\mathbf{F}}_p} (\sigma)$.

 Let $\text{ind}_K ^{G}\sigma$ be the smooth representation of $G$ compactly induced from $\sigma$, i.e., the representation of $G$ with underlying space $S(G, \sigma)$
\begin{center}
$S(G, \sigma)=\{f: G\rightarrow W\mid  f(kg)=\sigma (k)\cdot f(g), \forall~k\in K, g\in G,~ \text{smooth~with~compact~support}\}$
\end{center}
and $G$ acting by right translation.

The spherical Hecke algebra $\mathcal{H}(K, \sigma)$ is defined as $\text{End}_G (\text{ind}^G _K \sigma)$. By \cite[Proposition 5]{B-L94}, it is isomorphic to the convolution algebra $\mathcal{H}_K (\sigma)$:
\begin{center}
$\mathcal{H}_K (\sigma)= \{\varphi: G\rightarrow \text{End}_{\overline{\mathbf{F}}_p}(\sigma) \mid \varphi(kgk')=\sigma(k)\varphi(g)\sigma(k'), \forall~k, k'\in K, g\in G, ~\text{smooth~with~compact~support}\}$
\end{center}

Let $\varphi$ be the function in $\mathcal{H}_K (\sigma)$, supported on $K\alpha K$ and satisfying $\varphi (\alpha)= j_\sigma$. Denote by $T$ the Hecke operator in $\mathcal{H}(K, \sigma)$, which corresponds to the function $\varphi$ via the isomorphism aforementioned between $\mathcal{H}_K (\sigma)$ and $\mathcal{H}(K, \sigma)$. We refer the readers to \cite[(4)]{X2016} for a formula of $T$.

The structure of the algebra $\mathcal{H}(K, \sigma)$ is well-understood, and it is isomorphic to $\overline{\mathbf{F}}_p [T]$ (\cite[Corollary 1.3]{Her2011a}).

\subsection{Definition of supersingular representations}\label{subsec: ss}

Let $\sigma$ be a weight of $K$, and $\chi_\sigma$ be the character of $I_K$ for its action on the line $\sigma^{I_{1, K}}$.

\subsubsection{Hecke eigenvalues of principal series for $T$}

In this part, we follow along the lines in \cite{B-L95} to compute the Hecke eigenvalue of a general principal series of $G$ for $T$.

\begin{lemma}\label{non-zero of hom to prin}
For a character $\varepsilon$ of $B$, the space $\textnormal{Hom}_G (\textnormal{ind}_K ^G \sigma, \textnormal{ind}^{G}_B \varepsilon)$ is at most one dimensional, and it is non-zero if and only if
\begin{center}
$\varepsilon \mid_{H\cap K} = \chi^{s}_\sigma$.
\end{center}

\end{lemma}

\begin{proof}
As we have an Iwasawa decomposition $G= BK$, the argument is identical to that of \cite[Proposotion 23]{B-L94}.
\end{proof}

\medskip
When the condition in Lemma \ref{non-zero of hom to prin} is satisfied, the space
\begin{center}
$\textnormal{Hom}_G (\textnormal{ind}_K ^G \sigma, \textnormal{ind}^{G}_B \varepsilon)$
\end{center}
is one dimensional, and thus it affords a character of the spherical Hecke algebra $\mathcal{H}(K, \sigma)$.

\begin{proposition}\label{eigenvalue for T}
The eigenvalue of $T$ on the space $\textnormal{Hom}_G (\textnormal{ind}_K ^G \sigma, \textnormal{ind}^{G}_B \varepsilon)$ is:
\begin{center}
$\varepsilon (\alpha)+ \sum_{(x, t)\in L^\times _{q^{4-t_K}}}\chi_\sigma (h(t))$
\end{center}
\end{proposition}

\begin{proof}
The argument of \cite[Proposition 3.24]{X2014} for $K_0$ can be slightly modified to work for any $K$.
\end{proof}

We will say the weight $\sigma$ is \emph{degenerate}, if 
\begin{center}
$\sum_{(x, t)\in L^\times _{q^{4-t_K}}}\chi_\sigma (h(t))\neq 0$,
\end{center}
Otherwise, we say $\sigma$ is \emph{regular}. Note that when $\chi_\sigma = 1$, it is trivial to see the above sum is equal to $-1$.

\begin{remark}\label{value of the sum in T_sigma}
The value of the sum $\sum_{(x, t)\in L^\times _{q^{4-t_K}}}\chi_\sigma (h(t))$ appearing in Proposition \ref{eigenvalue for T} is known, and the readers are referred to \cite[Appendix A]{Karol-Peng2012} for a full list of it. In terms of loc.cit, the \textbf{degenerate} case will only happen in two situations:

$1)$.~ $\chi_\sigma$ is of the trivial type.

$2)$.~ $\chi_\sigma$ is hybrid and $K$ is hyperspecial.

\medskip
When $\sigma$ is degenerate, we have (loc.cit):
\begin{center}
$\sum_{(x, t)\in L^\times _{q^{4-t_K}}}\chi_\sigma (h(t))= -\chi_\sigma (h(\mathfrak{t}))$.
\end{center}

\end{remark}

\subsubsection{The Hecke operator $T_\sigma$}

On account of Proposition \ref{eigenvalue for T}, we consider $T_\sigma \in \mathcal{H}(K, \sigma)$ given by:
\begin{center}
$T_\sigma :=T- \sum_{(x, t)\in L^\times _{q^{4-t_K}}}\chi_\sigma (h(t))$
\end{center}

\begin{definition}\label{super reps}
An irreducible smooth representation $\pi$ of $G$ is called supersingular if it is a quotient of $\textnormal{ind}^G _K \sigma/(T_\sigma)$, for some weight $\sigma$ of $K$.
\end{definition}

\begin{remark}
Recall from \cite[III, Part A]{AHHV17}, an irreducible admissible smooth representation $\pi$ of a $p$-adic reductive group $G$ is \textbf{supersingular}, if it contains a weight $\sigma$ (w.r.t a special parahoric subgroup $K$), satisfying that the space
\begin{center}
$\textnormal{Hom}_G (\textnormal{ind}^G _K \sigma, \pi)$
\end{center}
admits a \textbf{supersingular} Hecke eigenvalue for the center $\mathcal{Z}_G (\sigma)$ of the spherical Hecke algebra $\mathcal{H}(K, \sigma)$. Here, the assumption of admissibility, as far as we understand, can be replaced by the existence of Hecke eigenvalue. In our case $G= U(2, 1)$,

$a)$.~the existence of Hecke eigenvalue is proved in \cite{X2018}.

$b)$.~a character of $\mathcal{H}(K, \sigma)$ is \textbf{supersingular} if and only if it kills the maximal ideal $(T_\sigma)$ (\cite[III 4, Corollary]{AHHV17}).

\end{remark}

\subsection{The space $(\textnormal{ind}_K ^G \sigma)^{I_{1, K}}$}

Let $\sigma$ be a weight of $K$. Fix a non-zero vector $v_0\in \sigma^{I_{1, K}}$. Let $f_n$ be the function in $(\textnormal{ind}_K ^G \sigma)^{I_{1, K}}$, supported on $K \alpha^{-n} I_{1, K}$, such that
\begin{center}
$f_n (\alpha^{-n})=  \begin{cases}
\beta_K\cdot v_0, ~~~~~~~n>0,\\
v_0 ~~~~~~~~~~~~~~~~~n\leq 0.
\end{cases}$
\end{center}
Then, we have the following (\cite[Lemma 3.5]{X2016})

\begin{lemma}

The set of functions $\{f_n \mid  n\in \mathbb{Z}\}$ consists of a basis of the $I_{1, K}$-invariants of $\textnormal{ind}^{G} _K\sigma$.

\end{lemma}

\section{$\text{dim}_{\overline{\mathbf{F}}_p} \pi^{I_{1, K}}\geq 2$ for supersingular $\pi$: a conjecture}\label{sec: conj}

\subsection{Some motivational remarks}\label{subsec: motivation}
The simple modules of the pro-$p$-Iwahori--Hecke algebra $\mathcal{H}(I_{1, K}, 1)$ of $G$ were classified in \cite{Karol-Peng2012}, and they are either one or two dimensional. The so-called simple \emph{supersingular} modules are all one-dimensional characters, and all non-supersingular simple modules arise from the $I_{1, K}$-invariants of irreducible subquotients of principal series. The goal of this paper is to prove that the subspace of $I_{1, K}$-invariants of certain supersingular representation $\pi$ is at least two dimensional.

To address our result, we remind the readers of a simple but important fact from classical case (i.e., complex case). The following is \cite[Proposition 4.3, (2)]{BH2006}:

\begin{proposition}\label{simple fact}
For a locally profinite group $\mathcal{G}$ and a compact open subgroup $\mathcal{K}$, the process $\pi \mapsto \pi^{K}$ induces a bijection between the following two sets:

$a)$.~equivalence classes of irreducible smooth representations $\pi$ of $\mathcal{G}$ such that $\pi^{\mathcal{K}}\neq 0$.

$b)$.~isomorphic classes of simple $\mathcal{H} (\mathcal{K}, 1)$-modules.
\end{proposition}

\medskip
As far as we know, the pro-$p$-Iwahori invariants of supersingular representations are only fully understood for the group $GL_2 (\mathbf{Q}_p)$ (\cite{Breuil03}) and some closely related cases, and in these cases the pro-$p$-Iwahori invariants are still simple modules; beyond $GL_2 (\mathbf{Q}_p)$, very little is known. For the group $GL_2 (\mathbf{Q}_{p^f})$ ($f> 1$), certain supersingular representation constructed by Breuil and Pa$\check{\text{s}}$k$\bar{\text{u}}$nas in \cite{BP12} has been proved in the recent work of Emerton--Gee--Savitt (\cite{EGS15}) that it appears in the cohomology of a suitable Shimura curve, and that its pro-$p$-Iwahori invariants matches the original recipe used to construct it. To our knowledge, that is the only case beyond $GL_2 (\mathbf{Q}_p)$ for a single supersingular representation its pro-$p$-Iwahori invariants is \emph{known}. In all, these work provide the first examples that the analogue of Proposition \ref{simple fact} above fail in the $p$-modular representation theory.

Our main input (Theorem \ref{main theorem for st}) proved in this paper, combined something natural, that is the pro-$p$-Iwahori invariants of an admissible supersingular representation \emph{only} admits supersingular subquotient \footnote{This is indeed a theorem of Ollivier--Vign$\acute{\text{e}}$ras \cite{OV2017} now, if one assumes admissibility.}, shows that the analogue of Proposition \ref{simple fact} fails wildly for supersingular representations of $G$, or possibly \emph{always} fails.

\subsection{$\text{dim}_{\overline{\mathbf{F}}_p} \pi^{I_{1, K}}\geq 2$ for any supersingular $\pi$ ?}

We propose the following:

\begin{conjecture}\label{main conjeture}
Let $\sigma$ be a weight of $K$. Then, for any irreducible quotient $\pi$ of $\textnormal{ind}^G _K \sigma/ (T_\sigma)$, the images of the functions $f_0$ and $f_1$ in $\pi$ are \textbf{not} proportional.
\end{conjecture}

We end this part by some general but more related comments.

$1)$.~The images of the functions $f_0$ and $f_1$ are linearly independent in $\textnormal{ind}^G _K \sigma/(T_\sigma)$ (\cite[Corollary 2.10]{X2017}).

$2)$.~An analogue of Conjecture \ref{main conjeture} for $GL_2$ holds and is indeed implicitly verified in \cite{B-L94}. See Theorem \ref{GL_2} in Section \ref{sec: appendix for GL_2}.

$3)$.~Our strategy to prove Conjecture \ref{main conjeture} is naturally by contradiction. Assume there is some $c\in \overline{\mathbf{F}}^\times _p$ such that
\begin{center}
the image of $f_0 + c\cdot f_1$ in $\pi$ is zero.
\end{center}
Then starting from such a function we proceed to generate more functions whose images in $\pi$ are still zero, and when $\sigma$ is the Steinberg weight we find a contradiction after we take some care in the process.

\section{The images of $(\textnormal{ind}_K ^G \sigma)^{I_{1, K}}$ under $S_K$ and $S_-$}\label{sec: the image of pro-p under two operators}

For a smooth representation $\pi$ of $G$, we have introduced two partial linear maps $S_K$ and $S_-$ in \cite[subsection 4.3]{X2017} as follows:
\begin{center}
$S_K: \pi^{N'_{m_K}} \rightarrow \pi^{N_{n_K}}$,

$v \mapsto \sum_{u \in N_{n_K}/N_{n_K +1}}~u \beta_K v$.
\end{center}

\begin{center}
$S_-: \pi^{N_{n_K}} \rightarrow \pi^{N'_{m_K}}$,

$v\mapsto \sum_{u'\in N' _{m_K}/N' _{m_K +1}}~u'\beta_K\alpha^{-1}v$
\end{center}
We have proved these two maps satisfy the following nice property:
\begin{proposition}\label{S_K and S_- preserve I_1}
If $v\in \pi^{I_{1, K}}$, then it is the same for $S_K v$ and $S_- v$.
\end{proposition}

\begin{proof}
This is \cite[(2) of Proposition 4.10]{X2017}.
\end{proof}

Our main goal in this section is to apply the maps $S_K$ and $S_-$ to the space $(\textnormal{ind}_K ^G \sigma)^{I_{1, K}}$ and determine explicitly the images. It is done in the following Proposition.

\begin{proposition}\label{the image of I_1 under S_K and S_-}
We have:

$(1)$.~For $n\geq 1$,
\begin{center}
$S_K f_n = f_{-n}$, ~$S_- f_n = c_- f_n$.
\end{center}
Here, the constant $c_-$ is given by:
\begin{center}
$c_-= \sum_{(x, t)\in L^\times _{q^{4-t_K}}}\chi_\sigma (h(t))$.
\end{center}

$(2)$.~For $n \geq 0$,
\begin{center}
$S_K f_{-n}= d_n f_{-n}$, ~$S_- f_{-n}= f_{n+1}$.
\end{center}
Here, the constant $d_n$ $(n\geq 1)$ is given by:
\begin{center}
$d_n =\sum_{(x, t)\in L^\times _{q^{t_K}}} \chi_\sigma ((h(t))$;
\end{center}
and the constant $d_0$ is equal to:
\begin{center}
$d_0= \begin{cases}
-\chi_\sigma (h(\mathfrak{t})), ~~\text{if}~\sigma\cong~\text{a twist of the Steinberg weight}
;\\
0, ~~~~~~~~~~~~~~~~~~\text{otherwise}.\end{cases}$
\end{center}
\end{proposition}

\begin{proof}
We will prove $S_K f_n = f_{-n}$ for $n\geq 1$ and $S_- f_{-n}= f_{n+1}$ for $n \geq 0$ at first.

For $n\geq 1$, the support of the function $S_K f_n$ is contained in:
\begin{center}
$K\alpha^{-n}I_{1, K}\beta_K N_{n_K} =K\alpha^n I_{1, K}$.
\end{center}
Then, by Proposition \ref{S_K and S_- preserve I_1} and \cite[Remark 3.8]{X2016}, the function $S_K f_n$ is proportional to $f_{-n}$. We compute:
\begin{center}
$S_K f_n (\alpha^n)= \sum_{u \in N_{n_K}/N_{n_K +1}}~f_n(\alpha^n u\beta_K)= f_n (\alpha^n \beta_K)= v_0$,
\end{center}
where we note that $\alpha^n u\beta_K \in K\alpha^n I_{1, K}$, for $u\in N_{n_K}\setminus N_{n_K +1}$ (\cite[Lemma 3.7 (3)]{X2017}). Hence, we have proved $S_K f_n= f_{-n} $ for $n\geq 1$.

\smallskip
For $n\geq 0$, the support of the function $S_- f_{-n}$ is contained in
\begin{center}
$K\alpha^n I_{1, K}\beta_K \alpha^{-1}N'_{m_K}= K\alpha^{-(n+1)}I_{1, K}$.
\end{center}
 By Proposition \ref{S_K and S_- preserve I_1} and \cite[Remark 3.8]{X2016} again, the function $S_- f_{-n}$ is proportional to $f_{n+1}$. We compute:
\begin{center}
$S_- f_{-n} (\alpha^{-(n+1)})= \sum_{u' \in N'_{m_K}/N'_{m_K +1}}~f_{-n} (\alpha^{-(n+1)}u'\alpha\beta_K)= \beta_K v_0$,
\end{center}
where we note that $\alpha^{-(n+1)}u'\alpha\beta_K \in K\alpha^{n+1} K$, for $u'\in N'_{m_K} \setminus N'_{m_K +1}$ (\cite[Lemma 3.7 (2), (3)]{X2017}). Thus, we have verified $S_- f_{-n}= f_{n+1}$, for $n\geq 0$.

\smallskip

We proceed to prove $S_- f_n =c_- f_n$ for $n\geq 1$; actually we will determine the value of $c_-$ explicitly. The support of the function $S_- f_n$ is contained in
\begin{center}
$K\alpha^{-n}I_{1, K} \alpha \beta_K N'_{m_K} \subseteq K \alpha^{n-1} I_{1, K} \cup K \alpha^n I_{1, K}$,
\end{center}
where the inclusion follows from \cite[Lemma 3.7 (1), (3)]{X2017}. We conclude that $S_- f_n \in \langle f_{-(n-1)}, f_n\rangle$ by Proposition \ref{S_K and S_- preserve I_1} and \cite[Remark 3.8]{X2016}. We compute:
\begin{center}
$S_- f_n (\alpha^{n-1})=\sum_{u'\in N' _{m_K}/N' _{m_K +1}} f_n (\alpha^{n-1}u'\alpha \beta_K)= \sum_{u'\in N' _{m_K}/N' _{m_K +1}} v_0 =0.$
\end{center}
It remains to compute $S_- f_n (\alpha^{-n})$:
\begin{center}
$S_- f_n (\alpha^{-n})=\sum_{u'\in N' _{m_K}/N' _{m_K +1}} f_n (\alpha^{-n} u'\alpha \beta_K)$.
\end{center}
Note that $\alpha^{-n} u'\alpha \beta_K \in K\alpha^{n-1}I_{1, K}$ for $u'\in N' _{m_K +1}$, and we are reduced to
\begin{center}
$S_- f_n (\alpha^{-n})=\sum_{u'\in (N' _{m_K}\setminus N' _{m_K +1})/N' _{m_K +1}} f_n (\alpha^{-n} u'\alpha \beta_K)$
\end{center}
For a $u' =n'(\ast, \varpi^{m_K} _E t)$ for some $t\in \mathfrak{o}^\times _E$, we have (using \cite[(1)]{X2016})
\begin{center}
$\alpha^{-n} u'\alpha \beta_K =n (\ast, \varpi^{2n-1+n_K} _E t^{-1})h(\bar{t}^{-1})\alpha^{-n}n'(\ast, \varpi^{m_K} _E t^{-1})$.
\end{center}
Thus, we immediately get:
\begin{center}
$S_- f_n (\alpha^{-n})=(\sum_{(x, t)\in L^\times _{q^{4-t_K}}}\chi_\sigma (h(t))) \beta_K v_0$,
\end{center}
here we have identified the group $N' _{m_K}/N' _{m_K +1}$ with $L _{q^{4-t_K}}$, via the map $L_{m_K}$ in the introduction.

Hence, we get
\begin{center}
$c_-= \sum_{(x, t)\in L^\times _{q^{4-t_K}}}\chi_\sigma (h(t))$.
\end{center}

\smallskip
We move to deal with the last statement: $S_K f_{-n}= d_n f_{-n}$, for $n\geq 0$. The support of the function $S_K f_{-n}$ is contained in
\begin{center}
$K\alpha^n I_{1, K}\beta_K N_{n_K} \subseteq K \alpha^n K$
\end{center}
By \cite[Remark 3.8]{X2016}, we get:

when $n=0$, $S_K f_0 \in \langle f_0\rangle$;

when $n >0$, $S_K f_{-n} \in \langle f_{-n}, f_n\rangle$.

We consider the second case at first. Assume $n >0$. We compute:
\begin{center}
$S_K f_{-n} (\alpha^{-n})= \sum_{u \in N_{n_K}/N_{n_K +1}}~f_{-n} (\alpha^{-n}u\beta_K)=\sum_{u \in N_{n_K}/N_{n_K +1}} \beta_K v_0= 0$.
\end{center}

Next, we compute $S_K f_{-n} (\alpha^n)$:
\begin{center}
$S_K f_{-n} (\alpha^n)= \sum_{u \in N_{n_K}/N_{n_K +1}}~f_{-n} (\alpha^n u\beta_K)$.
\end{center}
Note that $\alpha^n u\beta_K \in K\alpha^{-n}N'_{m_K}$, for $u\in N_{n_K +1}$. We are thus reduced to:
\begin{center}
$S_K f_{-n} (\alpha^n)= \sum_{u \in (N_{n_K}\setminus N_{n_K +1})/N_{n_K +1}}~f_{-n} (\alpha^n u\beta_K)$.
\end{center}
For $u= n(\ast, \varpi^{n_K} _E t)$, for some $t\in \mathfrak{o}^\times _E$, we have (using \cite[(1)]{X2016}):
\begin{center}
$\alpha^n u\beta_K= n'(\ast, \varpi^{2n-1+m_K} _E t^{-1})h(t)\alpha^n n(\ast, \varpi^{n_K} _E t^{-1})$.
\end{center}
Thus, we get
\begin{center}
$S_K f_{-n} (\alpha^n)=(\sum_{(x, t)\in L^\times _{q^{t_K}}} \chi_\sigma ((h(t))) v_0$,
\end{center}
here we have identified the group $N_{n_K}/N_{n_K +1}$ with $L _{q^{t_K}}$, via the map $L_{n_K}$ in the introduction. Hence, we get:
\begin{center}
$d_n =\sum_{(x, t)\in L^\times _{q^{t_K}}} \chi_\sigma ((h(t))$
\end{center}

\begin{remark}
The exact values of $c_-$ and $d_n$ ($n\geq 1$) depend on the nature of the character $\chi_\sigma$, and they have already been computed explicitly in \cite[Appendix A]{Karol-Peng2012}. Note that $c_-$ is the same sum we have used in subsection \ref{subsec: ss}.
\end{remark}

\medskip
We still need to compute the constant $d_0$ appearing in $S_K f_0 =d_0 f_0$. By definition, the constant $d_0$ is determined by
\begin{equation}\label{equation determines d_0}
\sum_{u \in N_{n_K}/N_{n_K +1}} u\beta_K v_0 =d_0 v_0.
\end{equation}

\medskip
We recall some stuff from \cite[section 5]{Karol-Peng2012}:

$1)$.~(Definition 5.2 of \emph{loc.cit})

To any character $\chi$ of $H_0 /H_1$, a subset $J_K (\chi)\subset \{s\}$ is attached.

$2)$.~(Definition 5.3 of \emph{loc.cit})

For any subset $J \subset J_K (\chi)$, one defines a character $M_{\chi, J}$ of the finite Hecke algebra $\mathcal{H}_{\Gamma_K} := \text{End}_{\Gamma_K} (\text{Ind}^{\Gamma_K} _\mathbb{U} 1)$.

$3)$.~(Proposition 5.4 of \emph{loc.cit})

Every simple module of the algebra $\mathcal{H}_{\Gamma_K}$ is isomorphic to $M_{\chi, J}$ for some character $\chi$ of $H_0 /H_1$ and some $J\subset J_K (\chi)$.

$4)$.~(Proposition 5.5 of \emph{loc.cit})

The functor $\sigma \rightarrow \sigma^{\mathbb{U}}$ gives a bijection of isomorphism classes of irreducible representations of $\Gamma_K$ and isomorphism classes of simple right $\mathcal{H}_{\Gamma_K}$-modules.

\medskip

By $4)$ above, we write our $\sigma$ as $\sigma_{\chi_\sigma, J}$ such that:
\begin{center}
 $\sigma^\mathbb{U} \cong M_{\chi_\sigma, J}$,
 \end{center}
for some $J\subset J_K (\chi_\sigma)$. Then, by comparing \eqref{equation determines d_0} and the right action of $\mathcal{H}_{\Gamma_K}$ (\cite[3.1, (1)]{Karol-Peng2012}) on $\sigma^{\mathbb{U}}$, we see immediately that
\begin{center}
$d_0= M_{\chi_\sigma, J} (T_{\beta_K})$,
\end{center}
where $T_{\beta_K}$ is the Hecke operator in $\mathcal{H}_{\Gamma_K}\hookrightarrow \mathcal{H} (I_{1, K}, 1)$ which corresponds to the double coset $I_{1, K}\beta_K I_{1, K}$. By the identification in \cite[Proposition 5.7]{Karol-Peng2012}, our statement for the value of $d_0$ now follows from the lists in Definition 5.3 of \emph{loc.cit}:
\begin{center}
$d_0= \begin{cases}
-\chi_\sigma (h(\mathfrak{t})), ~~\text{if}~\sigma\cong~\text{a twist of the Steinberg weight}
\\
0, ~~~~~~~~~~~~~~~~~~\text{otherwise}\end{cases}$
\end{center}
Here, we note that the element $\beta_K$ is different from a normalized one (i.e., with determinant $1$) used in \cite{Karol-Peng2012} by exactly the diagonal matrix $h(\mathfrak{t})$.
\end{proof}

\begin{remark}
Our argument for the value of $d_0$ in the Proposition is nearly formal and works for any $\sigma$. But we point out:

$1)$.~ It is trivial to see $d_0 =0$ if $\sigma$ is a character.

$2)$.~ In the only case that $d_0$ is non-zero, i.e., $\sigma$ is a twist of the Steinberg weight, we can work it out by hand without referring to \cite{Karol-Peng2012}.
\end{remark}
\medskip

\section{$\overline{f_1}\neq 0$ in the degenerate case}

When the weight $\sigma$ is \emph{degenerate}, the Hecke operator $T_\sigma$ is different from $T$ (Definition \ref{super reps}). Because of that, we may verify with ease that the image of the function $f_1$ in a supersingular $\pi$ is non-zero.

\begin{proposition}\label{nonvanishing of f_1}
Assume $\sigma$ is degenerate, and $\pi$ is an irreducible quotient of $\textnormal{ind}^G _K \sigma/(T_\sigma)$. Then,  we have $\overline{f_1}\neq 0$.
\end{proposition}

\begin{proof}
Recall that when $\sigma$ is degenerate, we have
\begin{center}
$T_\sigma= T +\chi_\sigma (h(\mathfrak{t}))$.
\end{center}

By the assumption, the image of the function $T_\sigma f_0$ in $\pi$ is zero. By \cite[Proposition 3.6]{X2016}, we have
\begin{center}
$T_\sigma f_0= f_{-1} + \lambda_{\beta_K, \sigma} f_1 +\chi_\sigma (h(\mathfrak{t})) f_0$,
\end{center}
where $\lambda_{\beta_K, \sigma}$ is some constant (\cite[Remark 2.1]{X2017}). If the image of $f_1$ in $\pi$ is zero, it is the same for the function $f_{-1}$, as by definitions (or by $(1)$ of Proposition \ref{the image of I_1 under S_K and S_-}) we have:
\begin{center}
$f_{-1}= S_K f_1$.
\end{center}

In all, if $\overline{f_1}= 0$, then we get $\overline{\chi_\sigma (h(\mathfrak{t})) f_0} =0$, i.e., $\overline{f_0}=0$, as $\chi_\sigma (h(\mathfrak{t}))\neq 0$. We get a contradiction, as the function $f_0$ generates $\textnormal{ind}^G _K \sigma$, and its image in the irreducible representation $\pi$ can not be zero.
\end{proof}

\begin{remark}
We have indeed verified that the image of $f_1$ is non-zero in any irreducible quotient of $\textnormal{ind}^G _K \sigma/ (T- \lambda)$ with $\lambda \neq 0$.
\end{remark}

\begin{remark}
For a supersingular representation $\pi$ containing a regular weight $\sigma$, we believe $\overline{f_1}$ is still non-zero in $\pi$. But it seems challenging (to the author) to prove it at this stage.
\end{remark}

\section{Proof of Theorem \ref{main theorem: intro} and Corollary \ref{non-simpleness: intro}}\label{sec: proof of main results}

\begin{theorem}\label{main theorem for st}
Assume $\sigma$ is the Steinberg weight. Then Conjecture \ref{main conjeture} holds.
\end{theorem}

\begin{proof}
Assume Conjecture \ref{main conjeture} fails, and there is a non-zero $c\in \overline{\mathbf{F}}^\times _p$ such that:
\begin{center}
the image $\overline{f(c)}$ of the function $f(c) :=f_0 + c\cdot f_1$ in $\pi$ is zero.
\end{center}
\medskip

Consider the function $f'(c) :=S_K f(c)$, which by Proposition \ref{the image of I_1 under S_K and S_-} is equal to
\begin{center}
$-f_0 + c\cdot f_{-1}$,
\end{center}
where we note that $d_0= -1$ in the current case.

Next, we consider the function $f'' (c) := S_- f'(c)$, which, by Proposition \ref{the image of I_1 under S_K and S_-} again, is equal to
\begin{center}
$-f_1 +c \cdot f_2$
\end{center}

Now, by \cite[Corollary 3.11]{X2016}, we have $T f_1 =f_2$ \footnote{Recall we have proved in \emph{loc.cit} that $T f_m = c_m f_m +f_{m+1}$ for any $m\geq 1$, even the value of the constant $c_m$ is not recorded explicitly there: it is zero if $\text{dim}_{\overline{\mathbf{F}}_p} \sigma > 1$ (using \cite[Remark 2.1, Lemma 3.7]{X2017}); when $\sigma$ is a character, it is equal to $\sum_{(x, t)\in L^\times _{q^{4-t_K}}}\chi_\sigma (h(t))$ (Remark \ref{value of the sum in T_sigma}).}. Recall that in the current case $T_\sigma =T+1$ (Definition \ref{super reps}). Putting all these together, we see
\begin{center}
$f'' (c) = (-c-1)f_1 + c\cdot (T+1)f_1$
\end{center}
As $\overline{f'' (c)}= \overline{(T+1)f_1}= 0$, we get
\begin{center}
$\overline{(-c-1)f_1}= 0$.
\end{center}
Thus, we must have $c=-1$, as $\overline{f_1}\neq 0$ (Proposition \ref{nonvanishing of f_1}).

Now we return to consider the function $S_- f(-1)$, which, by Proposition \ref{the image of I_1 under S_K and S_-}, is equal to
\begin{center}
$f_1 + f_1 =2f_1$,
\end{center}
where we note that in the current case $c_-$ is clearly equal to $-1$. We then get a contradiction, as $\overline{2f_1}\neq 0$ in $\pi$ ($p\neq 2$ !).
\end{proof}

\begin{remark}
Note that in Theorem \ref{main theorem for st} we don't put any restriction on the field $F$, e.g., it is not necessary to be characteristic $0$. But we do have used the assumption $p\neq 2$ !
\end{remark}

\begin{remark}\label{twist}
By Theorem \ref{main theorem for st}, Theorem \ref{main theorem: intro} follows immediately. As the kernel of any character of $G$ contains the group $I_{1, K}$, Theorem \ref{main theorem: intro} can be slightly extended to any supersingular representation of $G$ that contains a twist of the Steinberg weight by a character extendable to $G$.
\end{remark}

\begin{remark}
We remark briefly the cases beyond Theorem \ref{main theorem for st}:

$a)$.~For a supersingular representation $\pi$ containing a regular weight $\sigma$, we can also prove $\overline{f_0}$ and $\overline{f_1}$ are not proportional in $\pi$, if we \textbf{know} $\overline{f_1}\neq 0$ in $\pi$.

$b)$.~It seems the strategy we use to prove Theorem \ref{main theorem for st} might not simply work for a weight like \textbf{the trivial character} of $K$. But it seems one may still be able to conclude an analogue of Theorem \ref{main theorem: intro} for the trivial weight case from Theorem \ref{main theorem: intro}, by \textbf{changing of weight} in this setting (see \cite[Example 6.14]{Her2011b}, or \cite[Corollary 2.17]{Sch14}).
\end{remark}

\begin{corollary}\label{non-simpleness}
Let $\pi$ be a supersingular representation of $G$ containing the Steinberg weight of $K$. Then, the space $\pi^{I_{1, K}}$ as a right module over $\mathcal{H} (I_{1, K}, 1)$, is \textbf{not} simple.
\end{corollary}

\begin{proof}
 As simple modules of the pro-$p$-Iwahori--Hecke algebra $\mathcal{H} (I_{1, K}, 1)$ are finite dimensional, for our purpose we may assume\footnote{We would like to thank Karol Kozio{\l} for pointing out this.} $\pi$ is admissible, that is $\text{dim}~\pi^{I_{1, K}} < \infty$. By \cite[4.2]{Karol-Peng2012}, any supersingular module of $\mathcal{H} (I_{1, K}, 1)$ is a character. Thanks to Ollivier--Vign$\acute{\text{e}}$ras \cite[Theorem 5.3, Remark 5.2 (e)]{OV2017}, we know the pro-$p$-Iwahori invariants of an admissible supersingular representation only admits supersingular subquotients. Now the corollary follows from Theorem \ref{main theorem for st}.
\end{proof}

\section{Appendix A: some remarks on $GL_2 (F)$}\label{sec: appendix for GL_2}

In this appendix, we point out that an analogue of Conjecture \ref{main conjeture} is true for $GL_2 (F)$, due to Barthel--Livn$\acute{\text{e}}$. \emph{Our notations in this appendix are independent from other parts of this paper.}

\medskip

Let $F$ be a non-archimedean local field, with ring of integers $\mathfrak{o}_F$ and maximal ideal $\mathfrak{p}_F$, and let $k_F$ be its residue field of characteristic $p$. Let $G=GL_2 (F)$, $K= GL_2 (\mathfrak{o}_F)$ the maximal compact open subgroup of $G$, $Z= F^\times$ the center of $G$. Let $I$ be the standard Iwahori subgroup of $G$, and $I_1$ be the pro-$p$-Sylow subgroup of $I$. Let $K_1$ be the first congruence subgroup of $K$, so that $K/K_1 \cong GL_2 (k_F)$. Fix a uniformizer $\varpi$ in $F$. Denote by $\alpha$ and $\beta$ respectively the following two matrices:
\begin{center}
$\begin{pmatrix} 1  & 0  \\ 0  & \varpi
\end{pmatrix}, \begin{pmatrix} 0  & 1  \\ 1  & 0
\end{pmatrix}$
\end{center}
Put $\gamma= \alpha \beta$. Note that $\gamma^2 = \varpi I_2$  and $\gamma$ normalizes $I_1$.

Let $\sigma$ be an irreducible representation of $GL_2 (k_F)$ over $\overline{\mathbf{F}}_p$, inflated to $K=GL_2 (\mathfrak{o}_F)$ via the isomorphism $K/K_1 \cong GL_2 (k_F)$. Conversely, any irreducible smooth representation of $K$ over $\overline{\mathbf{F}}_p$ arises from such an inflation. Nowadays, such a representation $\sigma$ is usually called a \emph{weight} of $GL_2 (k_F)$ or $K$.  We extend $\sigma$ to a representation of $KZ$ by requiring $\varpi$ acts trivially.

An irreducible smooth $\overline{\mathbf{F}}_p$-representation of $G$ is called \emph{supersingular} if it is a quotient of $\text{ind}^G _{KZ} \sigma /(T)$, for a weight $\sigma$ of $K$, and for certain Hecke operator $T$ in the spherical Hecke algebra $\mathcal{H} (KZ, \sigma)$ (\cite[3.1]{B-L94}).

Fix a non-zero vector $v_0 \in \sigma^{I_1}$. For $n\in \mathbb{Z}$, let $\psi_n$ be the function in $(\text{ind}^G _{KZ} \sigma)^{I_1}$, supported on $KZ \alpha^{-n} I_1$, such that
\begin{center}
$\psi_n (\alpha^{-n})=  \begin{cases}
\beta\cdot v_0, ~~~~~~~n>0,\\
v_0 ~~~~~~~~~~~~~~~n\leq 0.
\end{cases}$
\end{center}

\begin{remark}
Up to a scalar, the functions $\psi_n$ defined above coincide that in \cite[section 4]{B-L94}. One can simply check
\begin{center}
$\psi_1 = \gamma \cdot \psi_0$
\end{center}
\end{remark}

Now we are ready to state the following analogue of Conjecture \ref{main conjeture} for $GL_2 (F)$, due to Barthel--Livn$\acute{\text{e}}$.

\begin{theorem}\label{GL_2}
Let $\sigma$ be a weight of $K$. Then, the images of $\psi_0$ and $\psi_1$ in any irreducible quotient of $\textnormal{ind}^G _{KZ} \sigma /(T)$ are \textbf{not} proportional.
\end{theorem}

\begin{proof}
The is indeed verified in the argument of \cite[Lemma 35]{B-L94}.
\end{proof}

\begin{remark}

Note that we can not deduce an analogue of Corollary \ref{non-simpleness} from Theorem \ref{GL_2}, as supersingular modules of $GL_2 (F)$ are all two-dimensional (\cite{Vig04}). Note also that, when $F= \mathbf{Q}_p$, a main input in \cite{Breuil03} is to prove the functions $\psi_0$ and $\psi_1$ together give a basis of the $I_1$-invariants of $\textnormal{ind}^G _{KZ} \sigma /(T)$, which fails for \textbf{any} other $F$.

\end{remark}

\section*{Acknowledgements}
The germ underlying this paper generates when the author was a Research Associate at University of Warwick, and he would like to thank Warwick Mathematics Institute for providing an outstanding working environment. The author was supported by a Postdoc grant from Leverhulme Trust RPG-2014-106 and European Research Council project 669655.
\medskip

\bibliographystyle{amsalpha}
\bibliography{new}

\texttt{Einstein Institute of Mathematics, HUJI, Jerusalem, 9190401, Israel}

\emph{E-mail address}: \texttt{Peng.Xu@mail.huji.ac.il}

\end{document}